\documentclass[a4paper,12pt,scrbook]{amsart}

\usepackage{fancybox}
\usepackage{fancyhdr} 
\usepackage{epigraph}
\usepackage{blindtext}
\usepackage{float}
\usepackage{amsmath,amssymb,amsthm}  	
\usepackage{eurosym}
\usepackage{graphicx} 				
\usepackage{epstopdf}
\usepackage[T1]{fontenc}
\usepackage[utf8]{inputenc}  			
\usepackage{polynom}
\usepackage{siunitx}
\usepackage{color}
\usepackage{lmodern}
\usepackage[headheight=20pt, headsep=40pt]{geometry}
\usepackage{nicefrac}
\usepackage{mathrsfs}
\usepackage{mathtools}

\def \ii{{\relax\ifmmode \mathrm{i} \else i \fi}}
\def \ee{{\relax\ifmmode \mathrm{e} \else e \fi}}
\def \dd{{\relax\ifmmode \mathrm{d} \else d \fi}}
\def \g{{\relax\ifmmode \mathrm{g} \else g \fi}}
\def \lg{{\relax\ifmmode \mathrm{lg} \else lg \fi}}

\def \N{\mathbb N}
\def \Z{\mathbb Z}

\newcommand*{\faktor}[2]{
  \raisebox{0.2\height}{\ensuremath{#1}}
  \mkern-3mu\diagup\mkern-3mu
  \raisebox{-0.2\height}{\ensuremath{#2}}
}
\renewcommand{\lg}[2]{\mathrm{lg}\left( \faktor{#1}{#2}\right)}

\theoremstyle{plain}
\newtheorem{satz}{Theorem}[section]

\newtheorem{lemma}[satz]{Lemma} 
\newtheorem{corollary}[satz]{Corollary}
\newtheorem{proposition}[satz]{Proposition}
\newtheorem*{wilf}{Wilf's Question}

\theoremstyle{definition}

\newtheorem{notiz}[satz]{Remark} 
\newtheorem{rem}[satz]{Remark}

\usepackage{hyperref}		

\title{Numerical Semigroups with $a_e = 2g+1$}   

\author{Michael Hellus}
\email{Michael.Hellus@mathematik.uni-regensburg.de}
\address{NWF I -Mathematik, Universität Regensburg,
93040 Regensburg, Germany}

\author{Reinhold H\"ubl}
\email{Reinhold.Huebl@dhbw.de}
\address{Zentrum f\"ur mathematisch-naturwissenschaftliches Basiswissen,
DHBW Mannheim, 68163 Mannheim, Germany}

\author{Anton Rechenauer}
\email{AntonRechenauer@googlemail.com}
\address{Regensburg, Germany}

\date{\today}

\begin{document}

\maketitle
\centerline{\textit{Dedicated to the memory of Rolf Waldi}}

\begin{abstract}
This article discusses numerical semigroups having a generator which is 
as large as possible. This turns out to be $2g+1$, where $g$ is the 
genus of the semigroup. We will show that these semigroups are closely 
related to symmetric semigroups and have interesting symmetry properties 
themselves. Furthermore we will show that Wilf's question has a positive 
answer for these semigroups and some semigroups derived thereof.
\end{abstract}

\section{Introduction}\label{sect:intro}

Numerical semigroups have been studied intensely since their first 
appearance in~\cite{Sy}, nevertheless many questions about them are still 
wide open. One particularly intriguing problem is a question of Wilf concerning 
the ratio of gaps and sporadic elements of such a semigroup (cf.~\cite{Wi}): 

For each such numerical semigroup $S$ let $F(S)$ be its Frobenius element, 
i.e. the largest integer not in $S$, let $g(S)$ be its genus, i.e. the number of 
positive integers not in $S$ (called the gaps of $S$) and let $e(S)$ be the 
number of minimal generators of $S$ (the embedding dimension of $S$). 

\begin{wilf}
For each numerical semigroup $S \neq \N$, is it true that
	\[ \frac {g(S)}{F(S)+1} \leq \frac {e(S)-1}{e(S)} \,\, ? \]
\end{wilf}

If $\sigma(S)$ denotes the number of integers $n < F(S)$ that 
are in $S$ (called the sporadic elements of $S$), then the question 
may be reformulated as

\begin{wilf}
For each numerical semigroup $S \neq \N$, is it true that 
	\[ F(S)+1 \leq e(S) \cdot \sigma(S) \,\, ? \]
\end{wilf}
 
As $F(S)+1 = \sigma(S) + g(S)$, this provides a relation between 
the number of gaps and the number of sporadic elements of a 
numerical semigroup. Though this question has a positive answer 
in a number of cases, it is still wide open in general 
(see~\cite{De} for a survey). 

In this note we study semigroups having a generator ``as large as 
possible'' which turns out to be $2 \cdot g(S)+1$. We show that 
these classes of semigroups have interesting symmetry properties and 
are closely related to symmetric semigroups, though with the exception 
of the multiplicity $2$ case they are not symmetric themselves. We also 
show that for these semigroups and some semigroups derived from them 
the answer to Wilf's question is positive. 

Numerical semigroups in general are closely related to curve singularities 
as they naturally appear as value semigroups of unibranch curve 
singularities. The semigroups studied here bear a close relationship  
to the classification of curve singularities by parametrization as 
curve singularities whose value semigroups satisfy $a_e = 2g+1$ show 
that the bounds on the $\mathcal{L}$--determinacy,  obtained 
in~\cite{HS}, are sharp in general.  

Parts of the research for this paper were done while the second author 
was visiting Purdue University. He expresses his gratitude to this 
institution for its support and hospitality. 

\section{Maximal possible generator}\label{sect:basic}

Throughout this paper $S = \langle a_1, \ldots, a_e\rangle$ 
will be a numerical semigroup with $a_1 < a_2 < \cdots < a_e$, 
embedding dimension $e$ (so that $a_1$, $\ldots$, $a_e$ are 
the minimal generators of $S$) and multiplicity $m = a_1$. By 
$g=g(S)$ we denote its genus and by $f = F(S)$ its Frobenius 
element. 

\begin{lemma}\label{lem:an_genus}
In the above situation
	$$ a_e \leq 2 g + 1 $$
\end{lemma}

\begin{proof} 
It suffices to show that any $\lambda \in\N$ with $\lambda  \geq 2 \cdot g + 2$ is 
contained in $S_+ + S_+$. 

Let $t = \lfloor \frac {\lambda}{2} \rfloor$. Then $t \geq g + 1$, and we have 
$t$ pairwise distinct presentations  
	$$ \begin{array} {l c l c l}
	\lambda & = & 1 & + & (\lambda - 1) \\
	\lambda & = & 2 & + & (\lambda - 2) \\
	& \vdots & & & \\
	\lambda & = & t & + & (\lambda - t) \\
	\end{array} $$
As $t > g$ for at least one equation
	$$ \lambda  = i  + (\lambda - i) $$
neither $i$ nor $\lambda-i$ can be gaps of $S$, i.e. both $i$ and $\lambda - i$ are in 
$S_+$, proving the lemma. 
\end{proof}

\begin{proposition}\label{prop:ae_maximal} 
Let $S = \langle a_1, \ldots , a_e \rangle$ be a numerical 
semigroup of genus $g$. Then the following are equivalent: 
\begin{enumerate}
    	\item $a_e = 2g+1$.
    	\item The assignment $\varphi$ given by  $\varphi(n) = 2g+1 - n$ 
	induces a well defined bijection $\varphi: \{1, \ldots, 2g\} \cap S 
	\longrightarrow \{1, \ldots, 2g \} \setminus S$.   
\end{enumerate}
\end{proposition}

\begin{proof}
First assume $a_e = 2g+1$. We get $g$ equations 
	$$\begin{array} {l c l c l}
	a_e & = & 1 & + & (a_e - 1) \\
	a_e & = & 2 & + & (a_e - 2) \\
	& \vdots & & & \\
	a_e & = & g & + & (a_e - g) \\
	\end{array} $$
As $a_e \notin S_++S_+$ in each equation at least one 
summand must be not in $S$. 
As there are only $g$ gaps, the existence and bijectivity 
of $\varphi$ follows. 

Assume now that a 
bijection $\varphi: \{1, \ldots, 2g\} \cap S 
\longrightarrow \{1, \ldots, 2g \} \setminus S$
is given by the assignment $\varphi$. Then 
for each $1 \leq n \leq g$ in the representation 
$2g+1 = n + (2g+1-n)$ exactly 
one summand is in $S$ and one is not. Thus $2g+1 \notin 
S_+ + S_+$. On the other hand, all $g$ gaps of $S$ are 
contained in $\{1, \ldots, 2g\}$ by the bijectivity of 
$\varphi$, hence $2g+1 \in S$. Thus it is a minimal 
generator of $S$, and by Lemma~\ref{lem:an_genus} 
necessarily the last one. 
\end{proof}

\begin{proposition}\label{prop:an_alm_symm}
Let  $S = \langle a_1, \ldots, a_e\rangle$ be a numerical semigroup of genus $g$. Then 
$a_e = 2g+1$ if and only if $S' = S \setminus \{ a_e\}$ is a symmetric semigroup.

Moreover in case $a_e = 2g+1$ we have: 
If $e=2$, then $a_1 = 2$ and $S' = \langle 2, a_e+2 \rangle$ and if 
$e > 2$, then $S' = \langle a_1, \ldots, a_{e-1}\rangle$. In both cases we have 
$m(S') = m(S)$. 
\end{proposition}

\begin{proof}
Assume $a_e = 2g+1$. 
As $a_e$ is a minimal generator of $S$, the set $S' = S \setminus \{ a_e\}$ certainly 
is a numerical semigroup with $g(S') = g+1$ and $F(S') = 2 g(S')-1 = 2g+1$. Hence 
$g(S') = \frac {F(S')+1}{2}$, implying that $S'$ is symmetric. 

Conversely suppose the $S' =S \setminus\{a_e\}$ is symmetric. Then 
$F(S') = 2g(S')-1 = 2g+1$.
If $a_e \neq F(S')$, then $2g+1 = F(S') = F(S) \leq 2g$, a contradiction.

If $e = 2$, $S = \langle m, l \rangle$ with $l \geq m+1 $, then $2 g+1 = 
(m-1) \cdot (l-1) + 1$ by the formula for the Frobenius element in 
numerical semigroups generated by two elements, hence $2g+1>l$ 
except in case $m = 2$ in which case $2g+1 = l$ 
and $S' = \langle a_1, l+2\rangle = \langle 2, a_e+2 \rangle$, 
and in case $e >2$ the additional statement is a 
consequence of the following lemma.  
\end{proof}

\begin{lemma}\label{lem:sym_max_gen}
Let $S'$ be a symmetric numerical semigroup with multiplicity $m \geq 3$. Then 
$S'$ is generated by its elements less than $F(S')$. 
\end{lemma}

\begin{proof}
As $S'$ is symmetric, $F(S') + 1 = 2g(S')$. Thus by Lemma~\ref{lem:an_genus} 
any element $s>F(S')+2$ cannot be part of a minimal generating set of $S'$. But by 
symmetry  (and as $m > 2$) $F(S')+1-m = F(S')-(m-1) \in S'$ and $F(S')+2-m 
= F(S')-(m-2) \in S'$, hence all generators are smaller than $F(S')$.
\end{proof}

\begin{rem}
The symmetric numerical semigroup 
$S'\coloneqq S\setminus\{a_e\}$ in Proposition \ref{prop:an_alm_symm}
is derived from an arbitrarily given numerical 
semigroup $S$ with $a_e=2g+1$ (notation from above). In fact this 
construction establishes a one-to-one correspondence between numerical 
semigroups $S$ having $a_e=2g+1$ and symmetric numerical semigroups $S'$, 
the inverse operation is given by $S\coloneqq S'\cup\{F(S')\}$, where $F(S')$ 
denotes the Frobenius number of $S'$.

This is obviuos by Proposition \ref{prop:an_alm_symm} and 
Lemma~\ref{lem:sym_max_gen}.
\end{rem}

\begin{corollary}\label{cor:an_frob}
Let $S = \langle a_1, \ldots, a_e\rangle$ be a numerical semigroup with multiplicity $m$ 
and with $a_e = 2g+1$. Then 
	$$ F(S) = a_e - m = 2g+1 - m $$
\end{corollary}

Let $S = \langle a_1, \ldots, a_e\rangle$ be a numerical semigroup with 
multiplicity $m = a_1$ and $n \in \N$ be a 
positive integer. The set 
	$$\operatorname{RG}(n, S)=\{L \in \{1,\ldots, n-1\} \vert L \notin S, n-L \notin S\}$$ 
is called the \textbf{set of reflected gaps} of $S$ with respect to $n$. By 
$\operatorname{Ap}(S)$ we denote the Ap\'ery--set of $S$, i.e. 
	$$ \operatorname{Ap}(S) = \{ s \in S \, \vert \,\, s - m \notin S \} $$
Note that $\operatorname{Ap}(S)$ contains exactly $m$ elements, that 
$\operatorname{Ap}(S) \cup \{m\}$ contains all minimal generators of $S$ 
and that for each Pseudo--Frobenius element $p$ of $S$ we have $p + m 
\in \operatorname{Ap}(S)$. Recall that an integer $p$ is called Pseudo--Frobenius 
element of $S$ if $p \notin S$ but $p+s \in S$ for all $s \in S_+$. 

\begin{proposition}\label{prop:gplusm_refelcted}
Let $S = \langle a_1, \ldots, a_e\rangle$ be a numerical semigroup with multiplicity 
$m=a_1$, genus $g = g(S)$ and with Frobenius element $f = F(S)$. Then the 
following are equivalent: 
\begin{itemize}
\item[i)] $a_e = 2 g+1$. 
\item[ii)] $m+\operatorname{RG}(f, S) = \operatorname{Ap}(S) 
\setminus \{0, f+m\}$. 
\item[iii)] $a_e = f+m$ and $\vert \operatorname{RG}(f, S) \vert = m-2$.
\end{itemize}
\end{proposition}

\begin{proof}
First assume that $a_e = f+m$. Then the following holds 

\begin{enumerate}
\item The map $\iota: \operatorname{RG}(f+m, S) \longrightarrow 
\operatorname{RG}(f, S)$ given by $\iota(L) = L-m$ is well--defined (and injective). 

In fact, if $L \in \operatorname{RG}(f+m, S)$, then $f+m-L \notin S$, hence 
$f+m-L  \leq f$, implying $L \geq m$. As $L$ is a gap, necessarily $L > m$ and 
therefore $L-m \in \{1, \ldots, f-1\}$. Certainly $L-m$ itself is a gap as $L$ is, and so is 
$f-(L-m)=f+m-L$. Hence $L-m \in \operatorname{RG}(f, S)$. 
\item $L \in \operatorname{RG}(f, S) \setminus \operatorname{im}(\iota)$ if and only 
if $L + m \in \operatorname{Ap}(S) \setminus \{0, f+m\}$. 

If $L \in \operatorname{RG}(f, S)$ but $L+m \notin  \operatorname{RG}(f+m, S)$, 
then $L$ is a gap and $f-L = (f+m) - (L+m)$ is a gap of $S$, hence necessarily 
$L+m$ is not a gap of $S$, thus $L + m \in \operatorname{Ap}(S)\setminus \{0, f+m\}$. 

Conversely if $L+m  \in \operatorname{Ap}(S) \setminus \{0, f+m\}$, then 
$L+m \in S$ but $L \notin S$. Furthermore $x=f-L = (f+m) - (L+m) \notin S$, as 
otherwise $f+m = x+(L+m) \in S_+$ contradicting the fact that $a_e = f+m$ is a minimal 
generator. Thus $L \in \operatorname{RG}(f, S) \setminus \operatorname{im}(\iota)$. 
\item $a_e = 2g+1$ if and only if $\operatorname{RG}(f+m, S) = \emptyset$

This is immediate by Proposition~\ref{prop:ae_maximal} as $a_e = f+m$.  
\end{enumerate}
\bigbreak

With this we are able to complete the proof of the proposition: 

Assume i). If $a_e = 2g+1$, then $a_e = f+ m$ by 
Corollary~\ref{cor:an_frob}, hence 
$\operatorname{RG}(f+m, S) = \emptyset$ by (3), and therefore 
	$$ m+\operatorname{RG}(f, S) 
	= \operatorname{Ap}(S) \setminus \{0, f+m\} $$ 
by (2), thus ii) holds.
 
Now assume ii), i.e. assume that $m+\operatorname{RG}(f, S) = 
\operatorname{Ap}(S) \setminus \{0, f+m\}$. 
As $\operatorname{Ap}(S) \setminus \{0, f+m\} $ has $m-2$ 
elements, it remains to show that $a_e = f+m$. 

Assume otherwise. Then $a_e < f+m$ (as any $s > f+m$ is in $S_++S_+$) and 
	$$ f+m = \sum\limits_{i=1}^e n_i \cdot a_i $$
with $n_i \geq 0$ and $\sum\limits_{i=1}^e n_i \geq 2$. As $f \notin S$, 
necessarily $n_1 = 0$, i.e. $ f+m = \sum\limits_{i=2}^e n_i \cdot a_i $, and as 
$a_e < f+m$, 
	$$ a_2, \ldots, a_e \in \operatorname{Ap}(S) \setminus \{0, f+m\} $$ 
Let $i_0 = \operatorname{min}\{i \, \vert \,\, n_i > 0\}$ and write 
	$ f+m = a_{i_0} + s $
with 
	$$ s = \left(n_{i_0}-1\right) \cdot a_{i_0} + \sum\limits_{i=i_0+1}^e n_i \cdot a_i 
	\in S_+ $$
As $a_{i_0} - m \in \operatorname{RG}(f, S)$ 
is a reflected gap of $S$ with respect to $f$,  
	$$ s = f + m - a_{i_0} = f - \left(a_{i_0} - m\right) \quad \notin S $$
a contradiction. Hence $a_e = f+m$ and $\vert \operatorname{RG}(f, S) \vert = m-2$ 
and iii) is true.

If iii) holds, i.e. if $a_e = f+m$ and $\vert \operatorname{RG}(f, S) \vert = m-2$, 
then by (2) we have that 
	$$ m + \operatorname{RG}(f, S) \setminus \operatorname{im}(\iota) 
	= \operatorname{Ap}(S) \setminus \{0, f+m\} $$ 
has $m-2$ elements. Hence $\operatorname{im}(\iota) = \emptyset$, thus 
$\operatorname{RG}(f+m, S) = \emptyset$, and therefore $a_e = 2g+1$ by (3), 
hence we get i). 

This completes the proof of the proposition.  

\end{proof}

\begin{notiz}
The condition $a_e = f+m$ is necessary but not sufficient for $a_e = 2g+1$: 

For $m \geq 3$ the numerical semigroups $S = \langle m, f+1, f+2, \ldots f+m\rangle$ 
with $f > m$ and $f$ not a multiple of $m$ do satisfy 
	$$ F(S) = f, \quad a_e = f + m $$
but $a_e = 2 \cdot g(S) + 1$ holds if and only if $f = m+1$. 

In fact if $a_e = 2g+1$, then $q=a_e-a_{e-1}$ is a  Pseudo-Frobenius element 
(see Lemma~\ref{lem:pseudofrob} below), satisfying $q <m$. Hence 
$q+m \in S$, and therefore $f < q+m < 2m$. Thus $g = f-1$, and therefore 
	$$ f+m = a_e = 2g+1 = 2f-1 $$ 
implying $f = m+1$. Conversely $S =  \langle m, m+2, m+3 \ldots 2m+1\rangle$ 
satisfies $g = m$ and $a_e = 2m+1 = 2g+1$. 
\end{notiz}

\begin{notiz}
The condition $\vert \operatorname{RG}(f, S) \vert = m-2$ is necessary but not 
sufficient for $a_e = 2g+1$ as the numerical semigroup $S = \langle 7, 11, 16, 17, 19 
\rangle$ shows: It has genus $13$, so $a_e < 2g+1$ but has the $m-2=5$ reflected 
gaps $5$, $8$, $10$, $12$ and $15$. 
\end{notiz}

\section{Pseudo-Frobenius elements and the canonical ideal}\label{sect:PF}

Let again $S = \langle a_1, \ldots, a_e\rangle$ be 
a numerical semigroup with $a_1 < \cdots < a_e$, 
genus $g = g(S)$, $a_e=2g+1$ and assume $m=a_1 \geq 2$.

\begin{lemma}\label{lem:pseudofrob}
The set of Pseudo--Frobenius elements of $S$ is given by 
	\[
	\operatorname{PF}(S) = \{a_e-a_1, \dots, a_e - a_{e-1} \}
	\]
\end{lemma}

\begin{proof}
By Proposition~\ref{prop:ae_maximal} every (positive) gap of $S$ is of the 
form $a_e-s$ for an element $s$ in $S$ with $s < a_e$. Such an element 
$a_e-s$ is a Pseudo-Frobenius element  if and only if  $(a_e-s)+t \in S$ 
for all $t \in S_+$, i.e. if and only if $a_e-(s-t) \in S$ for all $t \in S_+$.  
Using Proposition~\ref{prop:ae_maximal} again this is equivalent to $s-t \notin S$ 
for all $0<t<s$, $t \in S_+$, i.e. $s$ a minimal generator of $S$. 
\end{proof}

\begin{corollary}\label{cor:type_S}
A numerical semigroup $S = \langle a_1, \ldots, a_e\rangle$ with $a_e = 2g+1$ has type 
	$$ t(S) = e(S) - 1 $$
\end{corollary}

Recall that for a numerical semigroup $S$ a canonical ideal $K = K(S)$ of $S$ 
is an $S$--ideal $K \subseteq \Z$ such that  
	$$ K - (K-I) = I $$  
for each $S$--ideal $I\subseteq \Z$ (cf.~\cite{Ja} \S 5). 
A canonical ideal always exists and is unique up to 
translation. Hence there is exactly one canonical ideal $K$ 
of $S$ with minimal element $0$. 

\begin{rem}
By~\cite[\S 5]{Ja}, ~\cite[Prop. 3.2]{Ba} 
and~\cite[ex. 2.14]{RGS} for an $S$--ideal $K \subseteq \Z$ 
the following are equivalent: 
\begin{enumerate}
    \item $K$ is a canonical ideal of $S$ with minimal 
    element $0$.
    \item $K$ is a canonical ideal with $S \subseteq K 
    \subseteq \N$.
    \item $K = \{z\in\mathbb Z\,\mid\,F(S)-z\notin S\} $.
    \item The elements $F(S) - p$, where $p$ runs through 
    the Pseudo--Frobenius elements of $S$ form a minimal 
    generating set of $K$. 
\end{enumerate}
\end{rem}

If the numerical semigroup $S$ satisfies the condition $a_e=2g+1$, then there is the following description of its canonical ideal $K$:
\begin{proposition}
Let $S = \langle a_1, \ldots, a_e\rangle$ be a numerical semigroup with $a_e=2g+1$ and let $K$ be its canonical 
ideal with minimal element $0$. Then
\[K=\{s-a_1\,\mid\,s\in S\}\setminus\{-a_1\}\]
and as an $S$--ideal $K$ is generated by the elements
	\[ 0 = a_1-a_1, a_2-a_1, \ldots, a_{e-1}-a_1\]
\end{proposition}

\begin{proof}
The claim follows from the symmetry of the numerical semigroup 
$S'=S\setminus\{a_e\}$ in Proposition~\ref{prop:an_alm_symm} 
and the explicit calculation of the Pseudo--Frobenius elements in  Lemma~\ref{lem:pseudofrob}.
\end{proof}

\section{Estimating the number of gaps and Wilf's inequality}\label{sect:Wilf}

For numerical semigroups $S$ with $a_e = 2g+1$ the answer to 
Wilf's question is positive.

\begin{satz}[Wilf's inequality]\label{thm:an_wilf}
Let $S = \langle a_1, \ldots, a_e\rangle$ be a numerical semigroup with 
$a_e = 2g+1$, genus $g(S)$ and Frobenius element $F(S)$. Then 
	\begin{equation}\label{eq:wilf} 
	\frac {g(S)}{F(S)+1} \leq \frac {e-1}{e} 
	\end{equation}
\end{satz}

\begin{proof}
As $S$ has type $t(S) = e-1$, we get from~\cite[Theorem 20]{FGH} that 
	$$ e \cdot (F(S)+1-g(S)) \geq F(S) +1 $$
hence the claim. 
\end{proof}

\begin{rem}
In the situation of theorem~\ref{thm:an_wilf}, we have 
by corollary~\ref{cor:an_frob}, 
	$$ F(S) = 2\cdot g(S)+1-m(S) $$ 
hence~\eqref{eq:wilf} becomes 
 	\begin{equation}\label{eq:wilf2} 
	\frac {g(S)}{2 \cdot g(S)+2-m(S)} \leq \frac {e-1}{e} 
	\end{equation}
which is equivalent to 
 	\begin{equation}\label{eq:wilf2.1} 
	e \cdot g(S) \leq 2 \cdot e \cdot g(S) + 2 \cdot e - e \cdot m(S) - 
	2 \cdot g(S) - 2 + m(S)
	\end{equation}
hence to
    	\begin{equation}\label{eq:wilf2.2} 
	0 \leq (e-2) \cdot g(S) - (e-1) \cdot (m(S) - 2)
	\end{equation}
i.e. to
 	\begin{equation}\label{eq:wilf3} 
	(m(S)-2) \cdot (e-1) \leq  (e-2) \cdot g(S)
	\end{equation}
resp., if $e > 2$, to 
 	\begin{equation}\label{eq:wilf4} 
	g(S) \geq  (m(S)-2) \cdot \frac {e-1}{e-2} 
	\end{equation}
Now let $S' = S \setminus \{a_e\}$ be the symmetric semigroup associated with $S$. 
Then, excluding the case $m(S) = 2$, by Proposition~\ref{prop:an_alm_symm} 
we have $m(S') = m(S)$ and $e(S') = e-1$, 
and obviously $g(S') = g(S) +1$. Thus equation~\eqref{eq:wilf4} is equivalent to 
 	\begin{equation}\label{eq:wilf5} 
	g(S') \geq  1 +  (m(S')-2) \cdot \frac {e(S')}{e(S')-1} 
	\end{equation} 
for any symmetric semigroup $S' \neq \N$ (the case $m(S') = 2$ being obvious). 

Estimate~\eqref{eq:wilf5} is actually valid for any numerical semigroup $T$ 
satisfying the inequality $F(T) > m(T) \geq 2$ . 

In fact for any such $T$, all integers  $n$ with 
$1 \leq n \leq 2 \cdot m(T)-1$ are either gaps 
or minimal generators, thus 	
	$$ e(T) + g(T) \geq 2 \cdot m(T) - 1 $$
We first treat the case $m(T) > e(T)$. Then 
	$$ \begin{aligned}
	(2 \cdot m(T) - e(T) - 2)  \cdot (e(T)-1) &\,- (m(T) -2) \cdot e(T)  \\
	& = (e(T) -2) \cdot (m(T)-e(T) - 1) \\
	& \geq 0 
	\end{aligned} $$
and we conclude 
	$$ (2 \cdot m(T) - e(T) - 2)  \cdot (e(T)-1) \geq  (m(T) -2) \cdot e(T) $$
hence 
	\begin{equation}\label{eq:wilf6}  
	g(T) \geq 2 \cdot m(T) - e(T) - 1 \geq 1 + (m(T)-2) \cdot \frac {e(T)}{e(T)-1} 
	\end{equation}
In case $m(T) = e(T)$ we have 
	$$ \begin{array} {l c l} 
	m(T) \cdot (e(T)-1) & = & m(T) \cdot (m(T)-1) \\
	& > &  m(T)^2- m(T) - 1 \\
	& = & (m(T)-1) + (m(T)-2) \cdot m(T) \\
	& = & (e(T)-1) + (m(T)-2) \cdot e(T) \\
	\end{array} $$
and thus, as $F(T) > m(T)$, 
	$$ g(T) \geq m(T) > 1 + (m(T)-2) \cdot \frac {e(T)}{e(T)-1} $$
(even with strict inequality). Note however that in the remaining case $F(T) < m(T)$, 
i.e. $T = \langle m, m+1, \ldots, 2m-1 \rangle$, inequality~\eqref{eq:wilf6} does 
not hold for $m \geq 3$.

It is not clear to us whether estimate~\eqref{eq:wilf6} has any relation to 
Wilf's question for nonsymmetric $T$. 
\end{rem}

\section{Closing the largest gap}\label{sect:close_gap} 

A well established technique to study numerical semigroups is deriving them from 
well understood semigroups by closing suitable gaps. In section~\ref{sect:basic} we 
have seen that all numerical semigroups $S=\langle a_1, \ldots, a_e\rangle $ with 
$a_e = 2g+1$ arise from a symmetric semigroup by closing its largest gap which 
provided the basis of most of the results obtained so far. In this section we will 
pick up this theme by closing the largest gap of $S$ (i.e. the two largest gaps of 
a symmetric semigroup). Before doing so, we recall some facts and arguments, 
communicated to us by the late Rolf Waldi, that seem to be well known to experts;  
we include them here for lack of a suitable reference. 

Let $S$ be an arbitrary numerical semigroup. A set $D$ of gaps of $S$ is called 
distinguished if for any gap $a$ of $S$ there exists an element $s \in S$ with 
$a + s \in D$. We note that any distinguished set $D$ of $S$ necessarily contains 
the set $\operatorname{PF}(S)$ of all the Pseudo--Frobenius elements of $S$, and 
that $\operatorname{PF}(S)$ itself is a distinguished set of gaps of $S$. 

\begin{lemma}\label{lem:dist} 
If $D$ is a distinguished set of gaps of $S$, containing $d$ elements, then 
	$$ \frac {g(S)}{F(S)+1} \leq \frac {d}{d+1} $$ 
\end{lemma}

\begin{proof}
This is immediate by~\cite[Theorem 20]{FGH}, as $d \geq t$, the type of $S$. 

An easy direct argument goes as follows:  Set $D = \{g_1, \ldots, g_d\}$ and let 
$\{s_1, \ldots, s_n\}$ be the set of sporadic elements of $S$, i.e. the elements 
$s \in S$ with $s < F(S)$. Then by the very 
definition of distinguished, the gaps of $S$ are contained in the set 
	$$ M = \bigcup_{i=1}^d \{g_i-s_1,\ldots, g_i-s_n\} $$
which contains at most $d \cdot n$ elements. Thus 
	$$ \frac {F(S) +1 - g(S)}{g(S)} \geq \frac {n}{n \cdot d} = \frac {1}{d} $$
which is equivalent to the claim. 
\end{proof}

From now on for the remainder of this section let 
$S = \langle a_1, \ldots, a_e \rangle $ be a nontrivial numerical semigroup 
with $a_e = 2g+1$ and let $T = S \cup \{a_e-a_1\} $. As  
$a_e-a_1 = F(S)$ is the largest gap of $S$, this is again a numerical 
semigroup with $F(T) < a_e-a_1$. 

\begin{lemma}\label{lem:lgap_min_gen}
If $a_e > 2 \cdot a_1$, then $T$ is minimally generated by $a_1, \ldots, a_{e-1}, 
a_e-a_1$. In particular $T$ has embedding dimension $e(T) = e$. 
\end{lemma}

\begin{proof}
First note that 
	$$ T \cap \left(\{0, 1, \ldots, a_e\} \setminus \{a_e-a_1\} \right) = 
	S  \cap \left(\{0, 1, \ldots, a_e\} \setminus \{a_e-a_1\} \right) $$
hence any element $0 < a < a_e-a_1$ is a minimal generator of $S$ if and only if it 
is a minimal generator of $T$. 

The element $a_e - a_1$ is a minimal generator of $T$ as it is a gap of $S$, whereas 
$a_e \in T_+ + T_+$ is not. 

Now let $a_e-a_1 < a < a_e$. Clearly if $a \in S_+ + S_+$, then $a \in T_++T_+$. 
Conversely if $a \in T_+ + T_+$, $a = t_1 + t_2$, then $t_1 \geq a_1$ 
as $a_e > 2 \cdot a_1$, hence $t_2 
< a_e-a_1$ and vice versa, implying that $t_1, t_2 \in S_+$, hence $a \in S_+ + S_+$. 

Finally we note that by Lemma~\ref{lem:an_genus} resp. its proof any $a \in T$ 
with $a \geq a_e = 2g(S) + 1 = 2g(T)+3$ cannot be a minimal generator of $T$, 
completing the proof of the lemma. 
\end{proof}

\begin{satz}[Wilf's inequality]\label{thm:lgap_wilf}
Let $S = \langle a_1, \ldots, a_e \rangle $ be a numerical semigroup 
with $a_e = 2g+1$ and let $T = S \cup \{a_e-a_1\} $. Then $T$ 
satisfies Wilf's inequality 
	\begin{equation}\label{eq:wilf_lgap}
	\frac {g(T)}{F(T)+1} \leq \frac {e(T)-1}{e(T)}
	\end{equation}
\end{satz} 

\begin{proof}
If $e=1$, then $T=S=\N$, and there is nothing to show. Thus we may assume $e \geq 2$. 
If $a_e < 2 \cdot a_1$, then by Proposition~\ref{prop:an_alm_symm}
	\[S=\langle a_1,a_1+1,\dots,2a_1-1\rangle\]
hence 
	\[T=\langle a_1-1,a_1,\dots,2a_1-3\rangle\]
and~\eqref{eq:wilf_lgap} is true as equality. Hence we may assume $a_e > 2 \cdot a_1$. 
Then $e(T) = e$ by Lemma~\ref{lem:lgap_min_gen}, thus by Lemma~\ref{lem:dist} it 
suffices to show that 
	$$ D = \{a_e - 2 a_1, a_e - a_2, \ldots, a_e-a_{e-1} \} $$
is a distinguished set of gaps of $T$. 

By Lemma~\ref{lem:lgap_min_gen} the elements of $D$ certainly are gaps of $T$. 

If $t$ is a gap of $T$, then it is one of $S$ as well. Thus $a_e-t \in S$ by 
Proposition~\ref{prop:ae_maximal} and $0 < a_e -t < a_e$. Hence 
	$$ a_e -t = \sum\limits_{i=1}^{e-1} n_i \cdot a_i $$
for some $n_i \geq 0$ with $\sum\limits_{i=1}^{e-1} n_i > 0$. If $n_i = 0$ for all 
$i > 1$, then necessarily $n_1\geq 2$ (as otherwise $t = a_e-a_1$ which is not a gap), 
hence $t + (n_1-2) \cdot a_1 = a_e-2 \cdot a_1 \in D$, and if $n_{i_0}>0$ for some 
$i_0 > 1$, then $t + \sum\limits_{i\neq i_0} n_i \cdot a_i + (n_{i_0} - 1) \cdot a_{i_0} 
= a_e - a_{i_0} \in D$, showing that $D$ is a distinguished set of gaps, thus completing 
the proof of the theorem. 
\end{proof}

\begin{rem} 
In the situation of the proof of Theorem~\ref{thm:lgap_wilf}, the set of 
distinguished gaps considered is precisely the set of Pseudo--Frobenius 
elements of $T$: 

As $D$ is distinguished, all Pseudo--Frobenius elements of $T$ are 
contained in $D$. Conversely for any $t \in T_+$ we have 
$a_e - 2 \cdot a_1 + t \geq a_e - a_1 > F(T)$, hence 
$a_e - 2 \cdot a_1 + t \in T_+$. For $i > 1$ we have by 
Lemma~\ref{lem:pseudofrob} that $a_e-a_i$ is a Pseudo--Frobenius 
element of $S$, hence $a_e-a_i + s \in S \subseteq T$ for 
any $s \in S_+$ and $a_e-a_i+ a_e-a_1 \geq a_e-a_1 > F(T)$, 
hence also $a_e-a_i+ a_e-a_1 \in T$.
\end{rem}

\section{Value-Semigroups of Curve Singularities}

Numerical semigroups are closely related to the local rings  
of curves and appear naturally as value semigroups 
of unibranch curves singularities. In~\cite{GP} parametrizations 
$\varphi: P = k[[X_1, \ldots, X_N]] \longrightarrow k[[t]]$ 
(or more generally $\varphi: P = k[[X_1, \ldots, X_N]] 
\longrightarrow k[[t_1]]\times \cdots \times k[[t_r]]$) of  
completions of local rings of curves are studied.  
Two such parametrizations 
$\varphi$ and $\psi$ are called  \textbf{left-equivalent} or 
$\mathbf{\mathcal{L}}$\textbf{-equivalent} if there
exists a $\sigma \in \mathrm{Aut}_k(P)$ such
that $\varphi \circ \sigma^{-1} = \psi$.
A parameterization $\varphi : P = k[[X_1, \ldots, X_N]] 
\longrightarrow  k[[t_1]]\times \cdots \times k[[t_r]]$ is called 
$j$-$\mathcal{L}$-\textbf{determined}
if any other parameterization $\psi : P = k[[X_1, \ldots, X_N]] 
\longrightarrow  k[[t_1]]\times \cdots \times k[[t_r]]$
with $\varphi = \psi \pmod{\overline{\mathfrak m}^{j+1}}$
is $\mathcal{L}$-equivalent to $\varphi$. 

In~\cite[Corollary 12]{GP} it was shown that any parametrization 
$\varphi$ is $4 \cdot \delta(R)-1$-$\mathcal{L}$-determined, 
a bound which was improved in~\cite[Corollary 3.3]{HS}  to 
$2 \cdot \delta(R) + 1 = 2g+1$ for unibranch curve 
singularities, where 
$\delta(R) = \operatorname{lg}\left(\faktor{\overline{R}}{R}\right)$ 
denotes the singularity-degree of $R$ and where $g = g(S)$ 
is the genus of the value semigroup of $S$. This bound is sharp 
as a general bound and value semigroups $S$ of singularities 
attaining this bound satisfy $a_e =2g+1$. Conversely for any 
such $S$ its numerical semigroup ring is an example of a 
singularity attaining this bound. Part of the motivation for 
this paper stems from these studies and the question to 
understand those singularities for which the bound is sharp. 
Understanding these singularities then helps improving the 
bound for other classes. The following result 
generalizes~\cite[Theorem 2.1]{Ng20}, where it is 
shown that a plane curve singularity is 
$(c(R){-}1)$--$\mathcal{L}$--determined, where $c(R)$ is 
the conductor--degree of $R$, for unibranch singularities. 

\begin{satz}\label{cor:gorenstein_det}
If $(R, \mathfrak m, k)$ is a complete reduced and irreducible 
Gorenstein--curve--singularity with multiplicity $m \geq 3$ and 
conductor degree $c(R)$, then any parametrization 
$\varphi: k[[X_1, \ldots, X_n]]\longrightarrow \overline{R}$ is  
$(c(R){-}1)$--$\mathcal{L}$--determined (hence also 
$(c(R){-}1)$--left--right--determined in the sense of~\cite{GP}). 
\end{satz}

\begin{proof}
Let $b_1, \ldots, b_n$ be the Herzog--Kunz sequence of $R$ 
(as defined in~\cite{HHMM}) and let $V(R) = 
\langle a_1, \ldots, a_e\rangle$ be the value semigroup of $R$, 
which is symmetric as $R$ is Gorenstein. 
By~\cite[Corollary 2.5]{HHMM} we know $b_n \leq a_e$ and by Lemma~\ref{lem:sym_max_gen}, 
	$$a_e < F(V(R)) <  c(R)$$ 
As by~\cite{HS} 2.2 any parametrization $\varphi$ of $R$ is 
$(d_R{-}1)$--$\mathcal{L}$--determined where
 $d_R =\mathrm{min}\{c(R), b_n+1\}$, 
the claim follows. 
\end{proof}

\begin{rem}
That a curve singularity $R$ is $(c(R){-}1)$--$\mathcal{L}$--determined 
is obviously the best one can hope for. Thus it is somewhat surprising that 
for a symmetric numerical semigroup $S'$ the semigroup ring $k[[S']]$ 
obtains the best possible bound whereas for the semigroup 
$S = S' \cup \{F(S)\}$ obtained from $S'$ by closing its largest gap 
the semigroup ring $k[[S]]$ attains the worst possible bound. 
\end{rem}

\bigbreak

\bibliographystyle{plain}

\bigbreak

\bigbreak
\end{document}